\documentclass[a4paper,reqno]{amsart}
\usepackage[T1]{fontenc}
\usepackage{amssymb,amscd,amsthm,latexsym}
\allowdisplaybreaks[3]
\usepackage{graphicx}
\usepackage[british]{babel}
\usepackage[all]{xy}
\usepackage{color}
\usepackage{mathbbol}
\usepackage{mathabx}
\usepackage{calrsfs}
\usepackage{booktabs}
\usepackage{mathtools}

\theoremstyle{plain}
\newtheorem{theorem}[subsection]{Theorem}
\newtheorem{lemma}[subsection]{Lemma}
\newtheorem{proposition}[subsection]{Proposition}
\newtheorem{corollary}[subsection]{Corollary}

\theoremstyle{definition}
\newtheorem{definition}[subsection]{Definition}

\theoremstyle{remark}
\newtheorem{example}[subsection]{Example}

\newtheorem{remark}[subsection]{Remark}

\newdir{ >}{
 @{}*!/-10pt/@{>} }
\newdir{ |>}{
 @{}*!/-5.5pt/@{|}*!/-10pt/:(1,-.2)@^{>}*!/-10pt/:(1,+.2)@_{>} }
\newdir{>>}{{}*!/3.5pt/:(1,-.2)@^{>}*!/3.5pt/:(1,+.2)@_{>}*!/7pt/:(1,-.2)@^{>}*!/7pt/:(1,+.2)@_{>}}
\newdir{ >>}{{}*!/8pt/@{|}*!/3.5pt/:(1,-.2)@^{>}*!/3.5pt/:(1,+.2)@_{>}}
\newdir{>}{{}*:(1,-.2)@^{>}*:(1,+.2)@_{>}}
\newdir{<}{{}*:(1,+.2)@^{<}*:(1,-.2)@_{<}}

\def\ophalfsplitpullback{%
 \ar@{-}[]+R+<6pt,-1pt>;[]+RD+<6pt,-6pt>%
 \ar@{-}[]+D+<.5ex,-6pt>;[]+RD+<6pt,-6pt>}

\newcommand{\defn}[1]{\textbf{#1}}

\newcommand{\To}{\Rightarrow}
\newcommand{\ontop}[2]{\overset{\mathclap{#1}}{#2}}

\newcommand{\C}{\ensuremath{\mathbb{C}}}
\newcommand{\bi}[2]{\ensuremath{\lgroup #1 \;\, #2 \rgroup}}
\newcommand{\matriz}[4]{\ensuremath{\bigl\lgroup \begin{smallmatrix} #1 & #2 \\ #3 & #4 \end{smallmatrix}\bigr\rgroup}}
\newcommand{\Gp}{\ensuremath{\mathsf{Gp}}}
\newcommand{\Mon}{\ensuremath{\mathsf{Mon}}}
\newcommand{\HSLat}{\ensuremath{\mathsf{HSLat}}}
\newcommand{\Mag}{\ensuremath{\mathsf{Mag}}}

\newcommand{\Pt}{\ensuremath{\mathsf{Pt}}}
\newcommand{\SPt}{\ensuremath{\text{$\s$-$\Pt$}}}
\newcommand{\cod}{\ensuremath{\mathrm{cod}}}
\newcommand{\End}{\ensuremath{\mathit{End}}}

\newcommand{\V}{\ensuremath{\mathbb{V}}}

\newcommand{\N}{\ensuremath{\mathbb{N}}}
\newcommand{\s}{\ensuremath{\mathcal{S}}}

\newcommand{\iS}[1]{{\bf (iS#1)}}
\renewcommand{\S}[1]{{\bf (S#1)}}

\newcommand{\ito}{\dasharrow}

\begin{document}

\title[Intrinsic Schreier split extensions]{Intrinsic Schreier split extensions}

\author[Andrea Montoli]{Andrea Montoli}
\address{Dipartimento di Matematica ``Federigo Enriques'', Universit\`{a} degli
Studi di Milano, Via Saldini 50, 20133 Milano, Italy}
\thanks{The first author was partially supported by
the Programma per Giovani Ricercatori ``Rita Levi-Montalcini'', funded by the Italian government through MIUR}
\email{andrea.montoli@unimi.it}

\author[Diana Rodelo]{Diana Rodelo}
\address{Departamento de Matem\'atica, Faculdade de Ci\^{e}ncias e Tecnologia, Universidade
do Algarve, Campus de Gambelas, 8005-139 Faro, Portugal and CMUC, Department of Mathematics, University of Coimbra, 3001-501 Coimbra, Portugal}
\thanks{The second author acknowledges partial financial assistance by the Centre for Mathematics of the
University of Coimbra -- UID/MAT/00324/2013, funded by the Portuguese
Government through FCT/MEC and co-funded by the European Regional
Development Fund through the Partnership Agreement PT2020}
\email{drodelo@ualg.pt}

\author[Tim Van~der Linden]{Tim Van~der Linden}
\address{Institut de
Recherche en Math\'ematique et Physique, Universit\'e catholique
de Louvain, che\-min du cyclotron~2 bte~L7.01.02, 1348
Louvain-la-Neuve, Belgium}
\thanks{The third author is a Research
Associate of the Fonds de la Recherche Scientifique--FNRS}
\email{tim.vanderlinden@uclouvain.be}



\begin{abstract}
In the context of regular unital categories we introduce an intrinsic version of the notion of a \emph{Schreier split epimorphism}, originally considered for monoids.

We show that such split epimorphisms satisfy the same homological properties as Schreier split epimorphisms of monoids do. This gives rise to new examples of $\s$-protomodular categories, and allows us to better understand the homological behaviour of monoids from a categorical perspective.
\end{abstract}

\date{\today}

\maketitle

\section{Introduction}

Schreier extensions of monoids, introduced in~\cite{Redei}, have been studied by Patchkoria in~\cite{Patchkoria77a, Patchkoria77b} in connection with the cohomology of monoids with coefficients in semimodules. Indeed, Patchkoria's second cohomology monoids can be described in terms of Schreier extensions. Moreover, Schreier split extensions correspond actually to monoid actions~\cite{Patchkoria,BM-FMS}, where an action of a monoid~$B$ on a monoid $X$ is a monoid homomorphism from $B$ to the monoid $\End(X)$ of endomorphisms of~$X$. These split extensions turned out to have the classical homological properties of split extensions of groups, such as the Split Short Five Lemma (see~\cite{SchreierBook, BM-FMS2} for more details on these properties).

In order to better understand this phenomenon of a distinguished class of (split) extensions of monoids behaving as (split) extensions of groups, Schreier extensions of monoids have been studied from a categorical point of view. The category of groups is protomodular~\cite{Bourn1991}, while the category of monoids is not. This led to the study of the notion of protomodularity relativised with respect to a suitable class $\s$ of split epimorphisms, giving rise to the notion of an $\s$-protomodular category~\cite{S-proto}, having the category of monoids with Schreier split epimorphisms as a key example. Later, in~\cite{MartinsMontoliSH} it was shown that every J\'{o}nsson--Tarski variety of universal algebras~\cite{JT} is an $\s$-protomodular category with respect to (a suitable generalisation of) the class of Schreier split epimorphisms. These categories satisfy relative versions of the basic properties of protomodular categories.

However, this categorical description of the homological properties of Schreier extensions of monoids is not entirely satisfactory. The definition of a Schreier (split) extension is not categorical, because it depends crucially on the element-wise approach involving a \emph{Schreier retraction}, which is just a set-theoretical map (rather than a morphism of monoids). Moreover, for the same category, there may be several different classes of split epimorphisms that give rise to a structure of an $\s$-protomodular category: in the case of monoids, some of such different classes have been identified (see~\cite{SchreierBook} and~\cite{BM-Nine-lemma} for a description of these examples). For these reasons, the notion of an $\s$-protomodular category is able to capture only some of the (very strong) homological properties of Schreier split epimorphisms. Indeed, this notion covers other situations that are not so well-behaved.

The aim of the present paper is to give a characterisation of Schreier split epimorphisms in completely categorical terms, without using elements. In order to do that, we use \emph{imaginary morphisms}---in the sense of Bourn and Janelidze, see~\cite{AMO, DB-ZJ-2009, DB-ZJ-2009b, RVdL4}---for the categorical Schreier retractions. The advantage of obtaining this characterisation is two-fold. On one hand, this approach may allow a sharper categorical interpretation of the homological properties of Schreier extensions than the one obtained through the notion of an $\s$-protomodular category. On the other hand, our notion of \emph{intrinsic Schreier split extension}, being categorical, can be considered in a wider context than that of J\'{o}nsson--Tarski varieties, namely in regular unital~\cite{B0} categories (under some additional assumptions). This may allow us to develop a meaningful cohomology theory for regular unital categories, which on one hand extends the well-established cohomology theory for semi-abelian categories~\cite{Bourn-Rodelo}, and on the other hand interprets categorically Patchkoria's cohomology of monoids. This is material for future work.

\section{Schreier split extensions of monoids}\label{Schreier split extensions of monoids}
In this section we recall from~\cite{SchreierBook, BM-FMS2} the main definitions and properties concerning Schreier split extensions.

A split epimorphism of monoids $f$ with chosen section $s$ and kernel $K$
\begin{equation}\label{split ext of monoids}
 \xymatrix@!0@=5em{ K \ar@{ |>->}[r]_-{k} & (X,\cdot,1) \ar@<-.5ex>@{>>}[r]_-{f} & Y \ar@<-.5ex>[l]_-{s}}
\end{equation}
is called a \defn{Schreier split epimorphism} if, for every $x\in X$, there exists a unique element $a\in K$ such that $x=k(a)\cdot sf(x)$. Equivalently, if there exists a unique function $q\colon X\dashrightarrow K$ such that $x=kq(x)\cdot sf(x)$ for all $x\in X$. We emphasise the fact that $q$ is just a function (not necessarily a morphism of monoids) by using an arrow of type $\dashrightarrow$.

The uniqueness property may be replaced~\cite[Proposition~2.4]{BM-FMS2} by an extra condition on $q$: the couple $(f,s)$ is a Schreier split epimorphism if and only if
\begin{itemize}
\item[\S1] $x=kq(x)\cdot sf(x)$, for all $x\in X$;
\item[\S2] $q(k(a)\cdot s(y))=a$, for all $a\in K$, $y\in Y$.
\end{itemize}

\begin{remark}\label{homogeneous}Recall from~\cite{SchreierBook}
that Schreier split epimorphisms are also called \defn{right homogenous} split epimorphisms. A split epimorphism as in~\eqref{split ext of monoids} is called \defn{left homogenous} if, for every $x\in X$, there exists a unique element $a\in K$ such that $x=sf(x)\cdot k(a)$. It is called \defn{homogeneous} when it is both left and right homogeneous.
\end{remark}

\begin{proposition}\cite[Proposition~2.1.5]{SchreierBook} \label{basic pps}
Given a Schreier split extension as in~\eqref{split ext of monoids}, the following hold:
\begin{itemize}
\item[\S3] $qk=1_K$;
\item[\S4] $qs=0$;
\item[\S5] $q(1)=1$;
\item[\S6] $kq(s(y)\cdot k(a))\cdot s(y)=s(y)\cdot k(a)$, for all $a\in K$, $y\in Y$.
\end{itemize}
\end{proposition}

We say that a split epimorphism (with fixed section $s$) is a \defn{strongly split epimorphism}~\cite{B9} (see also~\cite{MartinsMontoliSobral2}, where the same notion was considered, in the regular context, under the name of \emph{regular point}) if its kernel~$k$ and section~$s$ form a jointly extremal-epimorphic pair $(k,s)$. It is \defn{stably strong}~\cite{2Chs} if every pullback of it along any morphism is a strongly split epimorphism (with the section induced by $s$).

\begin{lemma}\cite[Lemma~2.1.6]{SchreierBook} \label{Schreier => strong}
Any Schreier split epimorphism is strongly split.
\end{lemma}

It is easy to see that every strongly split epimorphism $(f, s)$ is such that~$f$ is the cokernel of its kernel, hence it gives rise to a split extension. The split extension~\eqref{split ext of monoids} is then called a \defn{Schreier split extension} and the map $q$ is called the associated \defn{(Schreier) retraction}. It is indeed a retraction, by \textbf{(S3)} above.

Actually any Schreier split epimorphism is stably strong, since Schreier split epimorphisms are stable under pullbacks:

\begin{proposition}\cite[Proposition~2.3.4]{SchreierBook}\label{Schreier stable for pbs}
Schreier split epimorphisms are stable under pullbacks along arbitrary morphisms.
\end{proposition}

\begin{corollary}\label{Schreier => stably strong}
Any Schreier split epimorphism is stably strong.
\end{corollary}

Some examples of Schreier split extensions are given by direct products:

\begin{proposition}\cite[Proposition~2.2.1]{SchreierBook}\label{Schreier products}
A split extension underlying a product of monoids
$$
 \xymatrix@!0@=5em{ X \ar@{ |>->}[r]_-{\langle 1_X, 0\rangle } & X\times Y \ar@<-.5ex>@{>>}[r]_-{\pi_Y} & Y \ar@<-.5ex>[l]_-{\langle0,1_Y\rangle}}
$$
is always a Schreier split extension.
\end{proposition}

\begin{corollary}\label{Schreier terminal and id}
Any terminal split extension and identity split extension
$$
 \xymatrix@!0@=4em{X \ar@{ |>->}[r]_-{1_X} & X \ar@<-.5ex>@{>>}[r]_-{!_X} & 0 \ar@<-.5ex>[l]_-{0_X}}
 \qquad\text{and}\qquad
 \xymatrix@!0@=4em{0 \ar@{ |>->}[r]_-{0_X} & X \ar@<-.5ex>@{>>}[r]_-{1_X} & X \ar@<-.5ex>[l]_-{1_X}}
$$
is a Schreier split extension.
\end{corollary}

Several other examples of Schreier split extensions are considered in~\cite{SchreierBook}. A~useful property of such split extensions is the following:

\begin{proposition}\cite[Lemma~4.1]{BM-FMS2}\label{Schreier compatibility of q}
Any morphism between two Schreier split extensions
$$
 \xymatrix@!0@C=5em@R=4em{ K \ar@{ |>->}@<-.5ex>[r]_-k \ar[d]_-{\widetilde{g}} & X \ar@<-.5ex>@{>>}[r]_-f \ar[d]_-g \ar@<-.5ex>@{-->}[l]_-q & Y \ar@<-.5ex>[l]_-s \ar[d]^-h \\
 K' \ar@{ |>->}@<-.5ex>[r]_-{k'} & X' \ar@<-.5ex>@{>>}[r]_-{f'} \ar@<-.5ex>@{-->}[l]_-{q'} & Y' \ar@<-.5ex>[l]_-{s'}}
$$
is compatible with respect to their retractions, i.e., $\widetilde{g}q=q'g$.
\end{proposition}

\section{$\s$-protomodular categories}

In this section we recall the definition of an $\s$-protomodular category, with respect to a class $\s$ of points in a category with a zero object.

Let $\C$ be a finitely complete category. A \defn{point} in $\C$ is a
split epimorphism~$f$ with a chosen section~$s$:
\[
\xymatrix@!0@=4em{X \ar@<-.5ex>@{>>}[r]_f & Y, \ar@<-.5ex>[l]_s}\qquad\qquad
fs=1_{Y}.
\]
We say that a point is \defn{(stably) strong} when it is a (stably) strongly split epimorphism. Consequently, its kernel $k$ and section $s$ form a jointly extremal-epimorphic pair $(k,s)$, and $(f,s)$ is part of a split extension.

We denote by $\Pt(\C)$ the category of points in $\C$, whose morphisms are pairs of morphisms which form commutative squares with both the split epimorphisms and their sections. The functor $\cod\colon \Pt(\C) \to \C$ associates with every split epimorphism its codomain. It is a fibration, usually called the \defn{fibration of points}. For each object $Y$ of $\C$, we denote by $\Pt_Y(\C)$ the fibre of this fibration, whose objects are the points with codomain $Y$.

Let $\s$ be a class of points in $\C$ which is stable under pullbacks along any morphism. If we look at it as a full
subcategory $\SPt(\C)$ of $\Pt(\C)$, then it gives rise to a subfibration $\s$-$\cod$ of the fibration of points.

\begin{definition}\cite[Definition~8.1.1]{SchreierBook}\label{S-protomodular category}
Let $\C$ be a pointed finitely complete category, and $\s$ a
pullback-stable class of points. We say that $\C$ is
\defn{$\s$-proto\-modular} when:
\begin{itemize}
\item[(1)] every point in $\SPt(\C)$ is a strong point;
\item[(2)]$\SPt(\C)$ is closed under finite limits in $\Pt(\C)$.
\end{itemize}
\end{definition}

\begin{example}\cite{SchreierBook}
The category $\Mon$ of monoids is $\s$-protomodular with respect to the class $\s$ of Schreier split epimorphism.
\end{example}

\begin{example}\cite{MartinsMontoliSH}
We recall that a variety of universal algebras is called a \defn{J\'onsson--Tarski} variety~\cite{JT} when its theory contains a unique constant $0$ and a binary operation $+$ such that $x+0=x=0+x$. So an algebra is a unitary magma, possibly equipped with additional operations. When treating the examples of groups and monoids, we shall write the binary J\'onsson--Tarski operation as a product. 

Every J\'onsson--Tarski variety is an $\s$-protomodular category with respect to the class of Schreier split epimorphisms. Indeed, the definition of a Schreier split epimorphism makes sense also in this wider context, and it gives rise to a whole family of examples of $\s$-protomodular categories.
\end{example}

\section{Imaginary morphisms}
The technique of \emph{imaginary morphisms} stems from the work of Bourn and Janelidze~\cite{AMO, DB-ZJ-2009, DB-ZJ-2009b}; it was further explored in \cite{RVdL4} by the second and third authors of the present article. Here we use imaginary morphisms in order to capture certain characteristic properties of J\'onsson--Tarski varieties and to define an intrinsic version of the concept of a Schreier retraction.

Here we assume that $\C$ is a regular category with enough (regular) projectives and that we can choose projective covers functorially. We write $\varepsilon_X\colon {P(X)\twoheadrightarrow X}$ for the chosen projective cover of some object $X$ in~$\C$: $\varepsilon_X$ is a regular epimorphism and $P(X)$ is a projective object, which means that for any morphism $z\colon P(X)\to Z$ and any regular epimorphism $f\colon{Y\twoheadrightarrow Z}$, there exists a morphism $y\colon {P(X)\to Y}$ in~$\C$ such that $fy=z$. In what follows, it will be convenient for us to let $P$ be part of a comonad $(P,\delta,\varepsilon)$; we say that $\C$ is equipped with \defn{functorial (comonadic) projective covers}. Note that for any morphism $f\colon X\to Y$ in $\C$
\begin{equation}\label{naturality of epsilon}
 f\varepsilon_X=\varepsilon_YP(f)
\end{equation}
and
\begin{equation}\label{naturality of delta}
 P^2(f)\delta_X=\delta_YP(f),
\end{equation}
where $P^2=PP$. Also
\begin{equation}\label{counit}
 \varepsilon_{P(X)} \delta_X=1_{P(X)}=P(\varepsilon_X)\delta_X
\end{equation}
and
\begin{equation}\label{coassociative}
 P(\delta_X)\delta_X = \delta_{P(X)}\delta_X,
\end{equation}
for all objects $X$ in $\C$.

\begin{example}\label{Varieties Comonad}
If $\V$ is a variety of universal algebras, then we may consider the free algebra comonad $(P,\delta,\varepsilon)$. For any algebra $X$, we have
$$
\begin{array}{rclcrcl}
\varepsilon_X \colon P(X) & \twoheadrightarrow & X & \mathrm{and} & \delta_X \colon P(X) & \hookrightarrow & P^2(X), \\
   \left[x\right] & \mapsto & x & &  [x] & \mapsto & \left[ [x] \right]
\end{array}
$$
where $[x]$ denotes the one letter word $x$, which are the generators of $P(X)$. In this case, any function $\overline{f}\colon{X\ito Y}$ between algebras $X$ and $Y$ extends uniquely to a morphism
$$
\begin{array}{rcl}
f \colon P(X) & \to & Y \\
   \left[x\right] & \mapsto & \overline{f}(x)
\end{array}
$$
in $\V$.	
\end{example}

This example motivates the following definition:

\begin{definition}\label{im}
A morphism $f\colon {P(X)\to Y}$ is called an \defn{imaginary morphism from $X$ to $Y$}; we write $\overline{f}\colon X\ito Y$.
\end{definition}

\begin{example}
In a variety of universal algebras $\V$ equipped with the free algebra comonad (Example~\ref{Varieties Comonad}), each function from an algebra $X$ to an algebra $Y$ may be considered as an imaginary morphism $X\ito Y$ in $\V$.	
\end{example}

An imaginary morphism $X\ito Y$ is not actually a morphism $X\to Y$ in~$\C$. Rather, it is a morphism in $\C$ with domain $P(X)$. Any real morphism $f\colon {X\to Y}$ may be considered as an imaginary morphism $\overline{f}\colon {X\ito Y}$, namely the composite $f \varepsilon_X$. In particular, $1_Y\colon Y\to Y$, considered as an imaginary morphism $Y\ito Y$, is $\varepsilon_Y\colon {P(Y)\to Y}$.

Actually, the imaginary morphisms are exactly the morphisms of the co-Kleisli category $\C_P$ induced by the comonad $(P,\delta,\varepsilon)$. There $\C_P(X,Y)=\C(P(X),Y)$ with the co-Kleisli composition. In the particular case where $\overline{f}\colon X\ito Y$ is an imaginary morphism and $g\colon Y\to Z$ and $h\colon W\to X$ are  morphisms in $\C$, they compose as in Figure~\ref{Figure Compositions}. 
\begin{figure}[h]
\begin{tabular}{c@{\qquad}c}
\toprule
imaginary composition & corresponding morphism in $\C$ \\
\hline
$\xymatrix{X \ar@{-->}[r]^-{\overline{f}} \ar@{-->}@/_1pc/[rr]_-{g\overline{f}} & Y \ar[r]^-g & Z}$ & $\xymatrix{P(X) \ar[r]^-f & Y \ar[r]^-g & Z}$ \\[5pt]
$\xymatrix{W \ar[r]^-h \ar@{-->}@/_1pc/[rr]_-{\overline{f}h} & X \ar@{-->}[r]^-{\overline{f}} & Y}$ & $\xymatrix{P(W) \ar[r]^-{P(h)} & P(X) \ar[r]^-f & Y}$\\
\bottomrule
\end{tabular}
\caption{Imaginary compositions}\label{Figure Compositions}
\end{figure}
The composition of the imaginary morphism $\overline{f}$ with the real morphism $g$ is precisely the composition of the imaginary morphism $\overline{f}$ with the corresponding imaginary morphism $\overline{g}=g\varepsilon_Y\colon P(Y) \to Z$ in the co-Kleisli category:
\[
\overline{g}\circ \overline{f}= g\varepsilon_YP(f)\delta_X \ontop{\eqref{naturality of epsilon}}{=}gf\varepsilon_{P(X)}\delta_X \ontop{\eqref{counit}}{=} gf.
\] 
Likewise,
\[
\overline{f}\circ \overline{h} = fP(h\varepsilon_W) \delta_W= fP(h) P(\varepsilon_W) \delta_W \ontop{\eqref{counit}}{=} fP(h).
\]

\begin{lemma}\label{Imaginary Splitting}
In a regular category with functorial projective covers, a morphism $f\colon {X\to Y}$ is a regular epimorphism if and only if it admits an imaginary splitting $\overline{s}\colon {Y\ito X}$. This means that $f \overline{s}=\overline{1_Y}\colon Y\ito Y$ or, equivalently, that $s\colon{P(Y)\to X}$ satisfies $f s=\varepsilon_Y$.
\end{lemma}
\begin{proof}
	This is an immediate consequence of the definitions.
\end{proof}

A pointed and regular category with binary coproducts is \defn{unital}~\cite{B0, Borceux-Bourn} when, for all objects $A$, $B$, the comparison morphism
\[
r_{A,B}=\matriz{	1_A}{0}{0}{1_B}\colon A+B\twoheadrightarrow A\times B
\]
is a regular epimorphism. If the category has functorial projective covers, then, by Lemma~\ref{Imaginary Splitting}, this is equivalent to saying that $r_{A,B}$ admits an imaginary splitting
\[
\overline{t_{A,B}}\colon {A\times B\ito A+B},
\]
i.e., there exists a morphism $t_{A,B}\colon P(A\times B)\to A+B$ such that
\begin{equation}\label{im splitting}
 r_{A,B} t_{A,B}=\varepsilon_{A\times B}.
\end{equation}

\begin{example}\label{LeftRightImaginarySplitting}
A variety of universal algebras $\V$ is unital if and only if it is a J\'{o}nsson--Tarski variety~\cite{Borceux-Bourn}. In this case, for any pair of algebras $(A,B)$ in $\V$, we make the following canonical choices of imaginary splittings for $r_{A,B}$: the \defn{direct imaginary splitting} $t^d$
\[
[(a,b)]\mapsto \underline{a}+\overline{b}
\]
which sends a generator $[(a,b)]\in P(A\times B)$ to the sum of $\underline{a}=\iota_A(a)$ with $\overline{b}=\iota_B(b)$ in~${A+B}$,
and the \defn{twisted imaginary splitting} $t^w$
\[
[(a,b)]\mapsto \overline{b}+\underline{a}
\]
which does the same, but in the opposite order. Note that each of those choices determines a natural transformation
\[
t\colon P((\cdot)\times (\cdot))\To (\cdot)+(\cdot)
\]
such that $r t=\varepsilon_{(\cdot)\times (\cdot)}$, where $r\colon (\cdot)+(\cdot)\To (\cdot)\times (\cdot)$ and
\[
\varepsilon_{(\cdot)\times (\cdot)}\colon P((\cdot)\times(\cdot))\To (\cdot)\times(\cdot).
\]
\end{example}

\begin{definition}
In a pointed regular (unital) category $\C$ with binary co\-products and functorial projective covers, a \defn{natural imaginary splitting (of the comparison from sum to product)} is a natural transformation
\[
t\colon P((\cdot)\times (\cdot))\To (\cdot)+(\cdot)
\]
such that $r t=\varepsilon_{(\cdot)\times (\cdot)}$.
\end{definition}

Note that, by Lemma~\ref{Imaginary Splitting}, the existence of a natural imaginary splitting in a pointed regular category with binary coproducts and functorial projective covers implies that this category is unital. Any J\'onsson--Tarski variety comes equipped with a direct and a twisted natural imaginary splitting $t^d$ and $t^w$ as in Example~\ref{LeftRightImaginarySplitting}. On the other hand, outside the varietal context there seems to be no reasonable way to characterise these cases categorically. We show next that in the category of monoids these are the only two possible choices.

\begin{proposition}\label{td and tw} In the category of monoids, any natural imaginary splitting $t$ must be the direct or the twisted natural imaginary splitting.
\end{proposition}
\begin{proof}
Consider the monoid of natural numbers $\N$. As a splitting of $\matriz{1_{\N}}{0}{0}{{1_\N}}$, we must have $t_{\N,\N}([(1,1)])=\underline{1}\overline{1}$ or $t_{\N,\N}([(1,1)])=\overline{1}\underline{1}$. 

Let $A$ and $B$ be arbitrary monoids and let us fix arbitrary elements $a\in A$ and $b\in B$. Consider the morphisms $f\colon \N \to A$, defined by $f(1)=a$, and $g\colon \N\to B$, defined by $g(1)=b$. The naturality of $t$ tells us that
\begin{center}
\begin{tabular}{ll}
    $t_{A,B}([(a,b)])=\underline{a}\overline{b}$, & when $t_{\N,\N}([(1,1)])=\underline{1}\overline{1}$; \\
    $t_{A,B}([(a,b))]=\overline{b}\underline{a}$, & when $t_{\N,\N}([(1,1)])=\overline{1}\underline{1}$.
\end{tabular}
\end{center}
Consequently, $t=t^d$ or $t=t^w$.
\end{proof}

A variation on this idea proves the same for unitary magmas. Then $\N$ must be replaced with the free unitary magma on a single generator $*$, which consists of non-associative words $()$, $*$, $(*,*)$, $(*,(*,*))$, $((*,*),*)$, etc. 

\begin{example}\label{GroupsDirectNorTwisted}
	In the category of groups, $P(A\times B)\to A+B\colon (a,b)\mapsto \underline{a}^{-1}\overline{b}\underline{a}^2$ determines a natural imaginary splitting which is neither direct nor twisted. This situation is studied in detail in the forthcoming article~\cite{SchreierObjects}.
\end{example}

\begin{remark}\label{remarks on t}
Any natural imaginary splitting $t\colon P((\cdot)\times (\cdot))\To (\cdot)+(\cdot)$ has the following properties:
\begin{itemize}
\item[1.] $t_{A,0}$ is isomorphic to $\varepsilon_A$
$$
\xymatrix@=2em{P(A) \ar@{>>}[r]^-{\varepsilon_A} \ar[d]|-{\cong} & A \ar@{=}[r]^-{1_A} \ar[d]|-{\cong} & A \ar[d]|-{\cong} \\
 P(A\times 0) \ar[r]^-{t_{A,0}} \ar@{>>}@/_1pc/[rr]_-{\varepsilon_{A\times 0}} & A+0 \ar@{>>}[r]^-{r_{A,0}} & A\times 0,}
$$
for all objects $A$ in $\C$;
\item[2.] the naturality of $t$ gives the commutative diagram
\begin{equation}\label{naturality of t}
\vcenter{\xymatrix{P(A\times B) \ar[r]^-{t_{A,B}} \ar[d]_-{P(u\times v)} & A+B \ar[d]^-{u+v} \\
 P(C\times D) \ar[r]_-{t_{C,D}} & C+D}}
\end{equation}
for all $u\colon A\to C$, $v\colon B\to D$ in $\C$;
\item[3.] from~\eqref{im splitting}, we deduce
\begin{equation}\label{pp1 for t}
 \bi{1_A}{0}t_{A,B}=\pi_A\varepsilon_{A\times B} \ontop{\eqref{naturality of epsilon}}{=}\varepsilon_A P(\pi_A)
\end{equation}
and
\begin{equation}\label{pp2 for t}
 \bi{0}{1_B}t_{A,B}=\pi_B\varepsilon_{A\times B} \ontop{\eqref{naturality of epsilon}}{=}\varepsilon_B P(\pi_B)
\end{equation}
for all objects $A$ and $B$ in $\C$;
\item[4.] using properties 1.\ and 2.\ above, we obtain the (regular epimorphism, monomorphism) factorisations
\begin{equation}\label{pp3 for t}
\vcenter{\xymatrix@C=4em@R=1em{P(A) \ar[r]^-{P(\langle 1_A,0\rangle)} \ar@{>>}[dr]_-{\varepsilon_A} & P(A\times B) \ar[r]^-{t_{A,B}} & A+B\\
 & A \ar@{ >->}[ru]_-{\iota_A}}}
\end{equation}
and
\begin{equation*}\label{pp4 for t}
\vcenter{\xymatrix@C=4em@R=1em{P(B) \ar[r]^-{P(\langle 0, 1_B\rangle)} \ar@{>>}[dr]_-{\varepsilon_B} & P(A\times B) \ar[r]^-{t_{A,B}} & A+B, \\
 & B \ar@{ >->}[ru]_-{\iota_B}}}
\end{equation*}
for all objects $A$ and $B$ in $\C$.
\end{itemize}
\end{remark}

\section{Intrinsic Schreier split extensions}\label{ISSE}
In this section we describe a categorical approach towards Schreier extensions. Here $\C$~will denote a regular unital category with binary coproducts, functorial projective covers and a natural imaginary splitting $t$.

\begin{definition} A split epimorphism $f$ with chosen section $s$ and kernel $K$
\begin{equation}\label{iSchreier}
\xymatrix@!0@=4em{ K \ar@{ |>->}[r]_-{k} & X \ar@<-.5ex>@{>>}[r]_-{f} & Y, \ar@<-.5ex>[l]_-{s}}
\end{equation}
is called an \defn{intrinsic Schreier split epimorphism} (with respect to $t$) if there exists an imaginary morphism $q\colon {X\ito K}$ (i.e. a morphism $q \colon P(X) \to K$), called the \defn{imaginary (Schreier) retraction}, such that
\begin{itemize}
\item[\iS1] $\bi{kq}{sf\varepsilon_X}t_{P(X),P(X)} P(\langle 1_{P(X)},1_{P(X)}\rangle)\delta_X=\varepsilon_X$, i.e., the diagram
$$
\xymatrix@C=5em{ P^2(X) \ar[r]^-{P(\langle 1_{P(X)},1_{P(X)}\rangle)} & P(P(X)\times P(X)) \ar[r]^-{t_{P(X),P(X)}} & P(X)+P(X) \ar[d]^-{\bi{kq}{sf\varepsilon_X}} \\
P(X) \ar@{ >->}[u]^-{\delta_X} \ar@{>>}[rr]_-{\varepsilon_X} & & X}
$$
commutes;
\item[\iS2] $qP(\bi{k}{s})P(t_{K,Y})\delta_{K\times Y}=\pi_K \varepsilon_{K\times Y}$, i.e., the diagram
$$
\xymatrix@C=4em{P^2(K\times Y) \ar[r]^-{P(t_{K,Y})} & P(K+Y) \ar[r]^-{P(\bi{k}{s})} & P(X) \ar[d]^-{q}\\
P(K\times Y) \ar@{ >->}[u]^-{\delta_{K\times Y}} \ar@{>>}[r]_-{\varepsilon_{K\times Y}} & K\times Y \ar@{>>}[r]_-{\pi_K} & K}
$$
 commutes.
\end{itemize}

\begin{proposition}\label{strong}
If the point $(f,s)$ in \eqref{iSchreier} admits $q\colon {X\ito K}$ satisfying \iS1, then it is a strong point.
\end{proposition}
\begin{proof}
From \iS1 we see that $\bi{kq}{sf\varepsilon_X}\colon$ $P(X)+P(X)\twoheadrightarrow X$ is a regular epimorphism. It easily follows that also $\bi{k}{s}\colon K+Y\twoheadrightarrow X$ is a regular epimorphism, thus $(k,s)$ is a jointly extremal-epimorphic pair.	
\end{proof}

We then call the point and split extension in~\eqref{iSchreier} an \defn{intrinsic Schreier point} and an \defn{intrinsic Schreier split extension}, respectively. The properties \iS1 and \iS2 are the respective translations of \S1 and \S2 to the ``imaginary'' context.
\end{definition}

\begin{proposition}\label{uniqueness of q}
For an intrinsic Schreier split extension~\eqref{iSchreier}, the imaginary retraction $q\colon {X\ito K}$ is unique.
\end{proposition}
\begin{proof}
Suppose that there exist two imaginary retractions $q,q'\colon P(X)\to K$ such that \iS1 and \iS2 hold. From \iS1 applied to $q$ we get
\begin{align}
& \varepsilon_X = \bi{kq}{sf\varepsilon_X}t_{P(X),P(X)}P(\langle 1_{P(X)},1_{P(X)}\rangle)\delta_X \nonumber \\
 \Leftrightarrow {} & \varepsilon_X =\bi{k}{s}(q+f\varepsilon_X)t_{P(X),P(X)}P(\langle 1_{P(X)},1_{P(X)}\rangle)\delta_X \nonumber \\
 \ontop{\eqref{naturality of t}}{\Leftrightarrow} {} & \varepsilon_X = \bi{k}{s}t_{K,Y}P(q\times (f\varepsilon_X))P(\langle 1_{P(X)},1_{P(X)}\rangle)\delta_X \nonumber \\
 \Leftrightarrow {} & \varepsilon_X = \bi{k}{s}t_{K,Y}P(\langle q,f\varepsilon_X\rangle)\delta_X.\label{iS1'}
\end{align}
Applying \iS1 to $q'$ we obtain a similar equality, namely
$$
 \bi{k}{s}t_{K,Y}P(\langle q,f\varepsilon_X\rangle)\delta_X = \bi{k}{s}t_{K,Y}P(\langle q',f\varepsilon_X\rangle)\delta_X.
$$
We use this equality and \iS2 applied to $q$ to obtain $q=q'$. Indeed
\begin{align*}
 & q P(\bi{k}{s})P(t_{K,Y})P^2(\langle q,f\varepsilon_X\rangle)P(\delta_X)\delta_X \\
 & \quad = q P(\bi{k}{s})P(t_{K,Y})P^2(\langle q',f\varepsilon_X\rangle)P(\delta_X)\delta_X \\
 \ontop{\eqref{coassociative}}{\Leftrightarrow} \quad {} & q P(\bi{k}{s})P(t_{K,Y})P^2(\langle q,f\varepsilon_X\rangle)\delta_{P(X)}\delta_X \\
 & \quad = q P(\bi{k}{s})P(t_{K,Y})P^2(\langle q',f\varepsilon_X\rangle)\delta_{P(X)}\delta_X \\
 \ontop{\eqref{naturality of delta}}{\Leftrightarrow} \quad {} & q P(\bi{k}{s})P(t_{K,Y})\delta_{K\times Y}P(\langle q,f\varepsilon_X\rangle)\delta_X \\
 & \quad = q P(\bi{k}{s})P(t_{K,Y})\delta_{K\times Y}P(\langle q',f\varepsilon_X\rangle)\delta_X \\
 \ontop{\iS2}{\Leftrightarrow} \quad {} & \pi_K\varepsilon_{K\times Y}P(\langle q,f\varepsilon_X\rangle)\delta_X = \pi_K\varepsilon_{K\times Y}P(\langle q',f\varepsilon_X\rangle)\delta_X \\
 \ontop{\eqref{naturality of epsilon}}{\Leftrightarrow} \quad {} & \pi_K\langle q,f\varepsilon_X\rangle \varepsilon_{P(X)}\delta_X = \pi_K\langle q',f\varepsilon_X\rangle \varepsilon_{P(X)}\delta_X \\
 \ontop{\eqref{counit}}{\Leftrightarrow} \quad {} & q= q'.\qedhere
\end{align*}
\end{proof}

The next results give the intrinsic versions of those recalled in Proposition~\ref{basic pps}.

\begin{proposition}\label{iretraction} Let $\C$ be a regular unital category with binary coproducts, functorial projective covers and a natural imaginary splitting $t$. If~\eqref{iSchreier} is an intrinsic Schreier split extension with imaginary retraction $q$, then:
\begin{itemize}
\item[\iS3] $qP(k)=\varepsilon_K$;
\item[\iS4] $qP(s)=0$;
\item[\iS5] $q0_{P(X)}=0_K$;
\item[\iS6] $\begin{aligned}[t]
\bi{s}{k} t_{Y,K}= {} & \bi{kqP(\bi{s}{k}) P(t_{Y,K})}{s\pi_Y\varepsilon_{Y\times K}\varepsilon_{P(Y\times K)}} t_{P^2(Y\times K),P^2(Y\times K)} \\
& \quad \circ P( \langle 1_{P^2(Y\times K)}, 1_{P^2(Y\times K)}\rangle)\delta_{P(Y\times K)}\delta_{Y\times K}.	
\end{aligned}$
\end{itemize}
\end{proposition}
\begin{proof}
If we compose each side of \iS2 with $P(\langle 1_K,0\rangle)$, use~\eqref{naturality of delta} and~\eqref{naturality of epsilon}, we obtain
\begin{align*}
 & qP(\bi{k}{s}) P(t_{K,Y})P^2(\langle 1_K,0\rangle)\delta_{K}=\pi_K \langle 1_K,0\rangle \varepsilon_{K} \\
 \ontop{\eqref{pp3 for t}}{\Leftrightarrow} \quad {} & qP(\bi{k}{s}) P(\iota_K)P(\varepsilon_K)\delta_{K}=\varepsilon_{K} \\
 \ontop{\eqref{counit}}{\Leftrightarrow} \quad {} & qP(k)=\varepsilon_K;
\end{align*}
this proves \iS3. Similarly, we prove \iS4 by composing each side of \iS2 with $P(\langle 0,1_Y\rangle)$; \iS5 is obvious.

Next, we prove a stronger equality from which \iS6 easily follows (by precomposing with $\delta_{Y\times K}$ and using~\eqref{counit}):
\begin{align*}
& \bi{s}{k}t_{Y,K}\varepsilon_{P(Y\times K)}
\ontop{\eqref{naturality of epsilon}}{=} \varepsilon_X P(\bi{s}{k} t_{Y,K}) \\
 \ontop{\iS1}{=} \; {} & \bi{kq}{sf\varepsilon_X} t_{P(X),P(X)} P(\langle 1_{P(X)}, 1_{P(X)}\rangle) \delta_X P(\bi{s}{k} t_{Y,K}) \\
 \ontop{\eqref{naturality of delta}}{=} \; {} & \bi{kq}{sf\varepsilon_X} t_{P(X),P(X)} P(\langle 1_{P(X)}, 1_{P(X)}\rangle) P^2(\bi{s}{k}t_{Y,K}) \delta_{P(Y\times K)} \\
 = \; {} & \bi{kq}{sf\varepsilon_X} t_{P(X),P(X)} P(\langle 1_{P(X)}, 1_{P(X)}\rangle P(\bi{s}{k} t_{Y,K})) \delta_{P(Y\times K)} \\
 = \; {} & \bi{kq}{sf\varepsilon_X} t_{P(X),P(X)} P( P(\bi{s}{k} t_{Y,K})\times P(\bi{s}{k} t_{Y,K}))\\
 & \quad \circ P(\langle 1_{P^2(Y\times K)}, 1_{P^2(Y\times K)}\rangle)\delta_{P(Y\times K)} \\
 \ontop{\eqref{naturality of t}}{=} \; {} &\bi{kqP(\bi{s}{k} t_{Y,K})}{sf\varepsilon_X P(\bi{s}{k} t_{Y,K})} t_{P^2(Y\times K),P^2(Y\times K)}\\
 & \quad \circ P( \langle 1_{P^2(Y\times K)}, 1_{P^2(Y\times K)}\rangle ) \delta_{P(Y\times K)}.
\end{align*}

To finish, we use
\begin{align*}
sf\varepsilon_XP( \bi{s}{k} t_{Y,K} ) & \ontop{\eqref{naturality of epsilon}}{=} sf\bi{s}{k}t_{Y,K}\varepsilon_{P(Y\times K)} \\
&= s\bi{1_Y}{0}t_{Y,K} \varepsilon_{P(Y\times K)} \\
 & \ontop{\eqref{pp1 for t}}{=} s\pi_Y\varepsilon_{Y\times K}\varepsilon_{P(Y\times K)}.\qedhere
\end{align*}
\end{proof}

Any binary product gives an example of an intrinsic Schreier split extension, like for Schreier split extensions for monoids (see Proposition~\ref{Schreier products}).

\begin{proposition}\label{binary product}
Any split extension given by a binary product is a Schreier split extension.
\end{proposition}
\begin{proof}
Given any split extension underlying a binary product in $\C$
$$
\xymatrix@!0@=6em{X \ar@{ |>->}[r]_-{\langle 1_X,0\rangle} & X\times Y \ar@<-.5ex>@{>>}[r]_-{\pi_Y} & Y \ar@<-.5ex>[l]_-{\langle 0, 1_Y\rangle}}
$$
we define $q=\pi_X\varepsilon_{X\times Y}\colon P(X\times Y)\to X$ for its imaginary retraction. To prove \iS1, we use
\begin{align*}
& \bi{\langle 1_X,0\rangle \pi_X\varepsilon_{X\times Y}}{\langle 0, 1_Y\rangle \pi_Y\varepsilon_{X\times Y}} t_{P(X\times Y),P(X\times Y)}\\
& \quad \circ P(\langle 1_{P(X\times Y)}, 1_{P(X\times Y)}\rangle)\delta_{X\times Y} \\
= {} & \bi{\langle \pi_X,0\rangle}{\langle 0, \pi_Y\rangle} (\varepsilon_{X\times Y}+\varepsilon_{X\times Y}) t_{P(X\times Y),P(X\times Y)}\\
& \quad \circ P(\langle 1_{P(X\times Y)}, 1_{P(X\times Y)}\rangle)\delta_{X\times Y} \\
\ontop{\eqref{naturality of t}}{=}{} & (\pi_X\times \pi_Y) \matriz{1_{X\times Y}}{0}{0}{1_{X\times Y}} t_{X\times Y, X\times Y} P(\varepsilon_{X\times Y}\times \varepsilon_{X\times Y})\\
& \quad \circ P(\langle 1_{P(X\times Y)}, 1_{P(X\times Y)}\rangle)\delta_{X\times Y} \\
 \ontop{\eqref{im splitting}}{=} {} & (\pi_X\times \pi_Y) \varepsilon_{X\times Y\times X\times Y} P(\langle \varepsilon_{X\times Y}, \varepsilon_{X\times Y}\rangle)\delta_{X\times Y} \\
 \ontop{\eqref{naturality of epsilon}}{=}{} & \varepsilon_{X\times Y}P(\pi_X\times \pi_Y) P(\langle \varepsilon_{X\times Y}, \varepsilon_{X\times Y}\rangle)\delta_{X\times Y} \\
 = {} & \varepsilon_{X\times Y}P(\varepsilon_{X\times Y})\delta_{X\times Y} \\
 \ontop{\eqref{counit}}{=} {} & \varepsilon_{X\times Y}.
\end{align*}
The proof of \iS2 is quite straightforward:
\begin{align*}
\pi_X\varepsilon_{X\times Y}P\bigl( \matriz{1_X}{0}{0}{1_Y} \bigr) P(t_{X,Y})\delta_{X\times Y} & \ontop{\eqref{im splitting}}{=} \pi_X\varepsilon_{X\times Y}P(\varepsilon_{X\times Y})\delta_{X\times Y}
 \ontop{\eqref{counit}}{=} \pi_X\varepsilon_{X\times Y}.\qedhere
\end{align*}
\end{proof}

\begin{corollary}\label{exs of Schreier}For any object $X$,
\[
\xymatrix@!0@=4em{X \ar@{ |>->}[r]_-{1_X} & X \ar@<-.5ex>@{>>}[r]_-{!_X} & 0 \ar@<-.5ex>[l]_-{0_X}}
\qquad\text{and}\qquad
\xymatrix@!0@=4em{0 \ar@{ |>->}[r]_-{0_X} & X \ar@<-.5ex>@{>>}[r]_-{1_X} & X \ar@<-.5ex>[l]_-{1_X}}
\]
are intrinsic Schreier split extensions.
\end{corollary}

\begin{proposition} \label{compatibility of q} Any morphism in $\Pt(\C)$ between intrinsic Schreier split extensions
$$
\xymatrix@!0@=4em{ K \ar@<-.5ex>@{ |>->}[r]_-k \ar[d]_-{\widetilde{g}} & X \ar@<-.5ex>@{-->}[l]_-q \ar@<-.5ex>@{>>}[r]_-f \ar[d]_-g & Y \ar@<-.5ex>[l]_-s \ar[d]^-h \\
 K' \ar@<-.5ex>@{ |>->}[r]_-{k'} & X' \ar@<-.5ex>@{-->}[l]_-{q'} \ar@<-.5ex>@{>>}[r]_-{f'} & Y' \ar@<-.5ex>[l]_-{s'}}
$$
is compatible with respect to their imaginary retractions, i.e., $\widetilde{g} q = q' P(g)$.
\end{proposition}

\begin{proof} We start by using~\eqref{iS1'} for the bottom intrinsic Schreier extension
\begin{align*}
& \bi{k'}{s'} t_{K',Y'} P(\langle q',f'\varepsilon_{X'}\rangle) \delta_{X'}P(g) = \varepsilon_{X'}P(g) \\
 \ontop{\eqref{naturality of delta},\eqref{naturality of epsilon}}{\Leftrightarrow} \quad {} & \bi{k'}{s'} t_{K',Y'} P(\langle q',f'\varepsilon_{X'}\rangle)P^2(g) \delta_{X} = g\varepsilon_{X} \\
 \ontop{\eqref{iS1'}}{\Leftrightarrow} \quad {} & \bi{k'}{s'} t_{K',Y'} P(\langle q'P(g),f'\varepsilon_{X'}P(g)\rangle) \delta_{X} = g\bi{k}{s} t_{K,Y} P(\langle q,f\varepsilon_{X}\rangle) \delta_{X} \\
 \ontop{\eqref{naturality of epsilon}}{\Leftrightarrow} \quad {} & \bi{k'}{s'} t_{K',Y'} P(\langle q'P(g),f'g\varepsilon_{X}\rangle) \delta_{X} = \bi{gk}{gs} t_{K,Y} P(\langle q,f\varepsilon_{X}\rangle) \delta_{X} \\
 \Leftrightarrow \quad {} & \bi{k'}{s'} t_{K',Y'} P(\langle q'P(g),f'g\varepsilon_{X}\rangle) \delta_{X} = \bi{k'}{s'} t_{K',Y'} P(\langle \widetilde{g}q,f'g\varepsilon_{X}\rangle) \delta_{X},	
\end{align*}
by using the commutativity of the diagram and~\eqref{naturality of t}.

We may now proceed as in the second part of the proof of Proposition~\ref{uniqueness of q} to conclude that $q'P(g)=\widetilde{g}q$.
\end{proof}

\begin{proposition} \label{sstrong} If the point $(f,s)$ in \eqref{iSchreier} admits $q\colon {X\ito K}$  satisfying \iS1, then it is a stably strong point.
\end{proposition}
\begin{proof}
We already know that $(f,s)$ is strong by Proposition~\ref{strong}. Now to see that $(f,s)$ is a stably strong point, we take its pullback along an arbitrary morphism~$g$
\begin{equation}\label{arbitraty pb}
\vcenter{\xymatrix@!0@R=4em@C=5em{ K \ar@{=}[r] \ar@{ |>->}[d]_-{\langle 0, k\rangle} & K \ar@{ |>->}[d]^-{k} \\
 Z\times_Y X \ar[r]^-{\pi_X} \ar@<-.5ex>@{>>}[d]_-{\pi_Z} \ophalfsplitpullback & X \ar@<-.5ex>@{>>}[d]_-{f} \\
 Z \ar[r]_-g \ar@<-.5ex>[u]_(.4){\langle 1_Z, sg\rangle} & Y. \ar@<-.5ex>[u]_-{s}}}
\end{equation}
To prove that $(\pi_Z, \langle 1_Z,sg \rangle)$ is strong it suffices to show the existence of an imaginary morphism $q'\colon {Z\times_Y X\ito K}$ satisfying \iS1. We define
\[
q'=qP(\pi_X)\colon {P(Z\times_Y X)\to K}
\]
and check that
\begin{multline*}
\varepsilon_{Z\times_Y X} = \bi{\langle 0, k\rangle q P(\pi_X)}{\langle 1_Z, sg\rangle \pi_Z\varepsilon_{Z\times_Y X}} \\
\circ t_{P(Z\times_Y X),P(Z\times_Y X)} P(\langle 1_{P(Z\times_Y X)}, 1_{P(Z\times_Y X)}\rangle)
 \delta_{Z\times_Y X}.		
\end{multline*}
Indeed, this equality follows from
\begin{align*}
&\pi_Z \bi{\langle 0, k\rangle q P(\pi_X)}{\langle 1_Z, sg\rangle \pi_Z\varepsilon_{Z\times_Y X}}\\
&\quad \circ t_{P(Z\times_Y X),P(Z\times_Y X)}
 P(\langle 1_{P(Z\times_Y X)}, 1_{P(Z\times_Y X)}\rangle)\delta_{Z\times_Y X} \\
 = {} & \bi{0}{\pi_Z\varepsilon_{Z\times_Y X}} t_{P(Z\times_Y X),P(Z\times_Y X)} P(\langle 1_{P(Z\times_Y X)}, 1_{P(Z\times_Y X)}\rangle)\delta_{Z\times_Y X} \\
 = {} & \pi_Z\varepsilon_{Z\times_Y X}\bi{0}{1_{P(Z\times_Y X)}}t_{P(Z\times_Y X),P(Z\times_Y X)} P(\langle 1_{P(Z\times_Y X)}, 1_{P(Z\times_Y X)}\rangle) \delta_{Z\times_Y X} \\
 \ontop{\eqref{pp2 for t}}{=}{} & \pi_Z\varepsilon_{Z\times_Y X}\varepsilon_{P(Z\times_Y X)}P(\pi_2)P(\langle 1_{P(Z\times_Y X)}, 1_{P(Z\times_Y X)}\rangle) \delta_{Z\times_Y X} \\
 = {} & \pi_Z\varepsilon_{Z\times_Y X}\varepsilon_{P(Z\times_Y X)}\delta_{Z\times_Y X}
 \ontop{\eqref{counit}}{=} \pi_Z \varepsilon_{Z\times_Y X}	
\end{align*}
and
\begin{align*}
& \pi_X \bi{\langle 0, k\rangle q P(\pi_X)}{\langle 1_Z, sg\rangle \pi_Z\varepsilon_{Z\times_Y X}}\\
& \quad \circ t_{P(Z\times_Y X),P(Z\times_Y X)}
 P(\langle 1_{P(Z\times_Y X)}, 1_{P(Z\times_Y X)}\rangle)\delta_{Z\times_Y X} \\
 = {} & \bi{kqP(\pi_X)}{sg\pi_Z\varepsilon_{Z\times_Y X}} \\
& \quad \circ t_{P(Z\times_Y X),P(Z\times_Y X)} P(\langle 1_{P(Z\times_Y X)}, 1_{P(Z\times_Y X)}\rangle)\delta_{Z\times_Y X} \\
 = {} &\bi{kqP(\pi_X)}{sf\pi_X\varepsilon_{Z\times_Y X}} \\
& \quad \circ t_{P(Z\times_Y X),P(Z\times_Y X)} P(\langle 1_{P(Z\times_Y X)}, 1_{P(Z\times_Y X)}\rangle)\delta_{Z\times_Y X} \\
 \ontop{\eqref{naturality of epsilon}}{=} {} &\bi{kq}{sf\varepsilon_X}(P(\pi_X)+P(\pi_X))\\
& \quad \circ t_{P(Z\times_Y X),P(Z\times_Y X)} P(\langle 1_{P(Z\times_Y X)}, 1_{P(Z\times_Y X)}\rangle) \delta_{Z\times_Y X} \\
 \ontop{\eqref{naturality of t}}{=} {} & \bi{kq}{sf\varepsilon_X}t_{P(X),P(X)} P(P(\pi_X)\times P(\pi_X))P(\langle 1_{P(Z\times_Y X)}, 1_{P(Z\times_Y X)}\rangle) \delta_{Z\times_Y X} \\
 = {} &\bi{kq}{sf\varepsilon_X} t_{P(X),P(X)} P(\langle P(\pi_X),P(\pi_X)\rangle) \delta_{Z\times_Y X} \\
 = {} &\bi{kq}{sf\varepsilon_X}t_{P(X),P(X)} P(\langle 1_{P(X)},1_{P(X)} \rangle) P^2(\pi_X)\delta_{Z\times_Y X} \\
 \ontop{\eqref{naturality of delta}}{=} {} &\bi{kq}{sf\varepsilon_X}t_{P(X),P(X)}P(\langle 1_{P(X)},1_{P(X)} \rangle)\delta_X P(\pi_X) \\
 \ontop{\iS1}{=} {} & \varepsilon_XP(\pi_X)
 \ontop{\eqref{naturality of epsilon}}{=} \pi_X \varepsilon_{Z\times_Y X}.\qedhere
\end{align*}
\end{proof}

We observe that the proof of the previous proposition actually tells us that the points $(f, s)$ satisfying \iS1 for a certain imaginary retraction $q$ are stable under pullbacks along any morphism.

Recall from~\cite{2Chs} that an object $Y$ is said to be a \defn{protomodular object} when all points over it are stably strong. Consequently, a finitely complete category is protomodular if and only if all of its objects are protomodular. It was shown there that the protomodular objects in $\Mon$ are precisely the groups. The next result gives a partial version of Corollary 3.1.7 from~\cite{SchreierBook} which says that a monoid $Y$ is a group if and only if all points over $Y$ are Schreier points.

\begin{corollary}\label{s-special => proto} If all points over $Y$ are intrinsic Schreier points, then $Y$ is a protomodular object.
\end{corollary}
\begin{proof} All points over $Y$ are stably strong by Proposition~\ref{sstrong}.
\end{proof}

The converse implication is false in general, as we will show in Section~\ref{special}.

We prove now that, if we apply our intrinsic definition to the category $\Mon$ of monoids, we regain the original definition of a Schreier split epimorphism (=~right homogeneous split epimorphism). Also left homogeneous and homogeneous split epimorphisms (see Remark~\ref{homogeneous}) fit the picture.

\begin{theorem}\label{Theorem Monoids}
In the case of monoids, the intrinsic Schreier split epimorphisms with respect to the direct imaginary splitting $t^d$ are precisely the Schreier split epimorphisms. 
Similarly, the intrinsic Schreier split epimorphisms with respect to the twisted imaginary splitting $t^w$ are the left homogeneous split epimorphisms. 

Hence the homogeneous split epimorphisms are the intrinsic Schreier split epimorphisms with respect to \emph{all} natural imaginary splittings. 
\end{theorem}
\begin{proof} Let~\eqref{split ext of monoids} be an intrinsic Schreier split extension of monoids with respect to $t^d$. Then there exists an imaginary retraction $q\colon P(X)\to K$ such that \iS1 and \iS2 hold, where $P(X)$ is the free monoid on $X$. We define the function
$$
\xymatrix@R=1pt@C=5pt{ X\; \ar@{-->}[rr]  \ar@{-->}@/^3.5ex/[rrrr]^-{q'} & & P(X) \ar[rr] && K \\
    x & \mapsto & [x] & \mapsto & q([x])}
$$
to be the Schreier retraction in $\Mon$. From \iS1 we prove \S1: for all $x\in X$,
\begin{align*}
& \bi{kq}{sf\varepsilon_X}t^d_{P(X),P(X)}P(\langle 1_{P(X)},1_{P(X)}\rangle)\delta_X([x]) = \varepsilon_X([x])\\
 \Leftrightarrow \quad {} & \bi{kq}{sf\varepsilon_X}t^d_{P(X),P(X)}P(\langle 1_{P(X)},1_{P(X)}\rangle)([[x]]) = x \\
 \Leftrightarrow \quad {} & \bi{kq}{sf\varepsilon_X}t^d_{P(X),P(X)}([([x],[x])]) = x \\
 \Leftrightarrow \quad {} & \bi{kq}{sf\varepsilon_X}(\underline{[x]}\cdot \overline{[x]}) = x \\
 \Leftrightarrow \quad {} & kq([x])\cdot sf(x) = x \\
 \Leftrightarrow \quad {} & kq'(x)\cdot sf(x) = x.	
\end{align*}
Similarly, from \iS2 we prove \S2: for all $a\in K$, $y\in Y$,
\begin{align*}
& qP(\bi{k}{s})P(t^d_{K,Y})\delta_{K\times Y}([(a,y)]) = \pi_K \varepsilon_{K\times Y}([(a,y)])\\
 \Leftrightarrow \quad {} & qP(\bi{k}{s})P(t^d_{K,Y})([[(a,y)]]) = \pi_K (a,y) \\
 \Leftrightarrow \quad {} & qP(\bi{k}{s})([\underline{a}\cdot\overline{y}]) = a \\
 \Leftrightarrow \quad {} & q([k(a)\cdot s(y)]) = a \\
 \Leftrightarrow \quad {} & q'(k(a)\cdot s(y)) = a.	
\end{align*}

Conversely, suppose that~\eqref{split ext of monoids} is a Schreier split extension of monoids with Schreier retraction $q'$. The map $q'$ determines a unique morphism of monoids $q\colon P(X)\to K$ given by
$$
 q([x_1x_2\cdots x_n])=q([x_1]\cdot[x_2]\cdot\;\cdots\;\cdot [x_n])=q'(x_1)\cdot q'(x_2)\cdot\;\cdots\; \cdot q'(x_n).
$$
In particular, on the generators we have $q([x])=q'(x)$, $\forall x\in X$, as above. Then~\eqref{split ext of monoids} is an intrinsic Schreier extension with respect to $t^d$. Indeed, to prove \iS1 and~\iS2 it suffices to check these equalities for the generators $[x]$, for all $x\in X$, and~$[(a,y)]$, for all $a\in K$ and $y\in Y$, respectively. They follow immediately from \S1 and \S2 and the fact that $q([x])=q'(x)$ for all $x\in X$, as for the previous implication.

The proof for left homogeneous split epimorphisms is similar: replace $t^d$ by~$t^w$. Thanks to Proposition~\ref{td and tw}, these are the only two natural imaginary splittings, which proves the final claim.
\end{proof}

\begin{remark}
It is not difficult to extend this result's first two statements (on left and right homogeneous split epimorphisms) to J\'onsson--Tarski varieties. If $\theta$ is an $n$-ary operation, we have
\[
q(\theta([x_1],\dots, [x_n]))=\theta(q'(x_1),\dots, q'(x_n)).
\]
The third statement is not valid, though, since according to Example~\ref{GroupsDirectNorTwisted}, in general these two cases do not cover all possible choices of a natural imaginary splitting $t$.
\end{remark}

\begin{remark}
The final statement in Theorem~\ref{Theorem Monoids} tells us that homogeneous split epimorphisms of monoids are in some sense ``more intrinsic'' than left or right homogeneous ones, since they do not depend on a choice of natural imaginary splitting, and thus admit a ``purely categorical'' description.
\end{remark}

\section{Stability properties}
The aim of this section is to show that any regular unital category $\C$ with binary coproducts, functorial projective covers and a natural imaginary splitting $t$ is an $\s$-protomodular category with respect to the class $\s$ of intrinsic Schreier split epimorphisms. First of all, we show that the class $\s$ is stable under pullbacks, which gives a subfibration of the fibration of points.

\begin{proposition}\label{stability pbs} Intrinsic Schreier split extensions are stable under arbitrary pullbacks.
\end{proposition}
\begin{proof}
Consider an intrinsic Schreier split extension and an arbitrary pullback as in~\eqref{arbitraty pb}. We know from Proposition~\ref{sstrong} that $q'=qP(\pi_X)\colon P(Z\times_Y X)\to K$ satisfies \iS1 for the split epimorphism $(\pi_Z,\langle 1_Z,sg\rangle)$. To prove \iS2 we calculate
\begin{align*}
&q P(\pi_X)P(\bi{\langle 0,k\rangle}{\langle 1_Z,sg\rangle})P(t_{K,Z}) \delta_{K\times Z} \\
= \; {} & q P(\bi{k}{sg})P(t_{K,Z}) \delta_{K\times Z} \\
= \; {} &q P(\bi{k}{s})P(1_K+g)P(t_{K,Z})\delta_{K\times Z} \\
\ontop{\eqref{naturality of t}}{=} \; {} & q P(\bi{k}{s}) P(t_{K,Y}) P^2(1_K\times g))\delta_{K\times Z} \\
\ontop{\eqref{naturality of delta}}{=} \; {} & q P(\bi{k}{s}) P(t_{K,Y})\delta_{K\times Y} P(1_K\times g) \\
 \ontop{\iS2}{=} \; {} & \pi_K \varepsilon_{K\times Y} P(1_K\times g) \\
 \ontop{\eqref{naturality of epsilon}}{=} \; {} & \pi_K (1_K\times g) \varepsilon_{K\times Z}
= \pi_K \varepsilon_{K\times Z}.	\qedhere
\end{align*}
\end{proof}

We have already explained that every intrinsic Schreier split epimorphism is a (stably) strong point (see Proposition~\ref{sstrong}). What remains to be shown---see Definition~\ref{S-protomodular category}---is that the full subcategory $\SPt(\C)$ of intrinsic Schreier points is closed in $\Pt(\C)$ under finite limits.

\begin{proposition} \label{stability products}Intrinsic Schreier split extensions are stable under binary pro\-ducts.
\end{proposition}
\begin{proof} Let
\[
\xymatrix@1@=3em{K \ar@{ |>->}[r]_-{k} & X \ar@<-.5ex>@{>>}[r]_-{f} & Y \ar@<-.5ex>[l]_-{s}}
\qquad
\text{and}
\qquad
\xymatrix@1@=3em{K' \ar@{ |>->}[r]_-{k'} & X' \ar@<-.5ex>@{>>}[r]_-{f'} & Y' \ar@<-.5ex>[l]_-{s'}}
\]
be intrinsic Schreier split extensions. Suppose that the imaginary retractions are $q\colon P(X)\to K$ and $q'\colon$ $P(X')\to K'$, respectively. We want to prove that
\[
\xymatrix@1@=3em{K\times K' \ar@{ |>->}[r]_-{k\times k'} & X\times X' \ar@<-.5ex>@{>>}[r]_-{f\times f'} & Y\times Y' \ar@<-.5ex>[l]_-{s\times s'}}
\]
is an intrinsic Schreier split extension. As imaginary retraction we use the imaginary morphism $\rho=(q\times q') \langle P(\pi_X), P(\pi_{X'})\rangle\colon P(X\times X')\to K\times K'$. For \iS1 we must prove that
\begin{multline*}
\varepsilon_{X\times X'}
=
\bi{(k\times k')\rho}{(s\times s')(f\times f')\varepsilon_{X\times X'}}\\
\circ t_{P(X\times X'),P(X\times X')}P(\langle 1_{P(X\times X')}, 1_{P(X\times X')}\rangle)\delta_{X\times X'}.
\end{multline*}
We have
\begin{align*}
&\pi_X \bi{(k\times k')\rho}{(s\times s')(f\times f')\varepsilon_{X\times X'}}\\
&\quad\circ t_{P(X\times X'),P(X\times X')} P(\langle 1_{P(X\times X')}, 1_{P(X\times X')}\rangle)\delta_{X\times X'} \\
 \ontop{\eqref{naturality of epsilon}}{=} {} &\bi{kqP(\pi_X)}{sf\varepsilon_XP(\pi_X)}t_{P(X\times X'),P(X\times X')}P(\langle 1_{P(X\times X')}, 1_{P(X\times X')}\rangle)\delta_{X\times X'} \\
 ={} & \bi{kq}{sf\varepsilon_X}(P(\pi_X)+P(\pi_X))\\
 &\quad \circ t_{P(X\times X'),P(X\times X')}P(\langle 1_{P(X\times X')}, 1_{P(X\times X')}\rangle)
 \delta_{X\times X'}\\
 \ontop{\eqref{naturality of t}}{=} {} &\bi{kq}{sf\varepsilon_X} t_{P(X),P(X)} P(P(\pi_X)\times P(\pi_X))P(\langle 1_{P(X\times X')}, 1_{P(X\times X')}\rangle)\delta_{X\times X'} \\
 ={} & \bi{kq}{sf\varepsilon_X} t_{P(X),P(X)} P(\langle P(\pi_X),P(\pi_X)\rangle)\delta_{X\times X'} \\
 = {} &\bi{kq}{sf\varepsilon_X} t_{P(X),P(X)} P(\langle 1_{P(X)},1_{P(X)} \rangle)P^2(\pi_X)\delta_{X\times X'} \\
 \ontop{\eqref{naturality of delta}}{=}{} & \bi{kq}{sf\varepsilon_X} t_{P(X),P(X)} P(\langle 1_{P(X)},1_{P(X)} \rangle)\delta_X P(\pi_X)\\
\ontop{\iS1}{=} {} &\varepsilon_X P(\pi_X)
 \ontop{\eqref{naturality of epsilon}}{=}\pi_X \varepsilon_{X\times X'} .	
\end{align*}
The proof that
\begin{multline*}
\pi_{X'} \varepsilon_{X\times X'}=
\pi_{X'} \bi{(k\times k')\rho}{(s\times s')(f\times f'\varepsilon_{X\times X'})}\\
\circ t_{P(X\times X'),P(X\times X')}
 P(\langle 1_{P(X\times X')}, 1_{P(X\times X')}\rangle)\delta_{X\times X'}
\end{multline*}
is similar.
For \iS2 we must show that
\[
 \rho P(\bi{k\times k'}{s\times s'}) P(t_{K\times K', Y\times Y'}) \delta_{K\times K'\times Y\times Y'}= \pi_{K\times K'} \varepsilon_{K\times K'\times Y\times Y'}.
\]
We have
\begin{align*}
& \pi_K \rho P(\bi{k\times k'}{s\times s'}) P(t_{K\times K', Y\times Y'})\delta_{K\times K'\times Y\times Y'} \\
 = {} & qP(\pi_X) P(\bi{k\times k'}{s\times s'}) P(t_{K\times K',Y\times Y'})\delta_{K\times K'\times Y\times Y'} \\
 = {} &q P( \bi{k}{s} (\pi_K+\pi_Y) t_{K\times K', Y\times Y'})\delta_{K\times K'\times Y\times Y'} \\
 \ontop{\eqref{naturality of t}}{=} {} & q P( \bi{k}{s} t_{K,Y} P(\pi_K\times \pi_Y) )\delta_{K\times K'\times Y\times Y'} \\
 = {} &q P( \bi{k}{s}) P( t_{K,Y}) P^2(\pi_K\times \pi_Y) \delta_{K\times K'\times Y\times Y'}\\
 \ontop{\eqref{naturality of delta}}{=} {} &q P( \bi{k}{s}) P( t_{K,Y})\delta_{K\times Y} P(\pi_K\times \pi_Y)\\
\ontop{\iS2}{=} {} &\pi_K\varepsilon_{K\times Y} P(\pi_K\times \pi_Y) \\
 \ontop{\eqref{naturality of epsilon}}{=} {} &\pi_K (\pi_K\times \pi_Y) \varepsilon_{K\times K'\times Y\times Y'}
 = \pi_K \pi_{K\times K'} \varepsilon_{K\times K'\times Y\times Y'}.	
\end{align*}
The proof that
\[
 \pi_{K'}\rho P(\bi{k\times k'}{s\times s'}) P(t_{K\times K', Y\times Y'})\delta_{K\times K'\times Y\times Y'} = \pi_{K'}\pi_{K\times K'} \varepsilon_{K\times K'\times Y\times Y'}
\]
is similar.
\end{proof}

\begin{proposition}\label{stability equalisers in Pt_CC}
Intrinsic Schreier split extensions in $\C$ are closed under equalisers in $\Pt(\C)$.
\end{proposition}
\begin{proof}
In the following diagram, the middle and right vertical extensions are intrinsic Schreier split extensions. Consider the morphisms $(g,u)$ and $(h,v)$ between them and the induced morphisms $\widetilde{g}$ and $\widetilde{h}$. Let $e$ be the equaliser of~$(g,h)$, $w$ be the equaliser of $(u,v)$, $\varphi$ and $\sigma$ be the induced morphisms and $L$ the kernel of $\varphi$
$$
\xymatrix@R=30pt@C=40pt{L \ar@{ >->}[r]^-{\widetilde{e}} \ar@{ |>->}[d]_-l & K \ar@<.5ex>[r]^-{\widetilde{g}} \ar@{ |>->}[d]_-k \ar@<-.5ex>[r]_-{\widetilde{h}} &
\ar@{ |>->}[d]_-{k'} K' \\
 E \ar@{ >->}[r]^-{e} \ar@<-.5ex>@{>>}[d]_-{\varphi} & X \ar@<.5ex>[r]^-g \ar@<-.5ex>[r]_-h \ar@<-.5ex>@{>>}[d]_-f & X' \ar@<-.5ex>@{>>}[d]_-{f'}\\
 W \ar@<-.5ex>[u]_-{\sigma} \ar@{ >->}[r]_-w & Y \ar@<-.5ex>[u]_-s \ar@<.5ex>[r]^-u \ar@<-.5ex>[r]_-v & Y'. \ar@<-.5ex>[u]_-{s'}}
$$
Since limits commute with limits, we know that the induced morphism $\widetilde{e}$ is the equaliser of $(\widetilde{g}, \widetilde{h})$. Moreover, $(e,w)$ is also the equaliser of $((g,u), (h,v))$ in~$\Pt(\C)$.

We want to prove that the left vertical extension above is an intrinsic Schreier split extension. Suppose $q\colon P(X)\to K$ and $q'\colon P(X')\to K'$ are the imaginary retractions for $(f,s)$ and $(f',s')$, respectively. The morphism $\gamma\colon P(E)\to L$,
\begin{equation*}\label{gamma}
\vcenter{\xymatrix{L \ar[r]^-{\widetilde{e}} & K \ar@<.5ex>[r]^-{\widetilde{g}} \ar@<-.5ex>[r]_-{\widetilde{h}} & K' \\
 P(E) \ar@{.>}[u]^-{\gamma} \ar[r]_-{P(e)} & P(X) \ar[u]_-q}}
\end{equation*}
which is induced by the universal property of the equaliser $\widetilde{e}$ because
\[
\widetilde{g} q P(e) \ontop{\ref{compatibility of q}}{=} q'P(g)P(e) = q' P(ge)=q'P(he) = \cdots =\widetilde{h} q P(e),
\]
gives the imaginary retraction with respect to $(\varphi, \sigma)$. We prove \iS1 by using that $e$ is a monomorphism:
\begin{align*}
&e\bi{l\gamma}{\sigma \varphi\varepsilon_E} t_{P(E),P(E)}P(\langle 1_{P(E)}, 1_{P(E)}\rangle)\delta_E\\
 = {} & \bi{el\gamma}{e \sigma \varphi\varepsilon_E}t_{P(E),P(E)}P(\langle 1_{P(E)}, 1_{P(E)}\rangle)\delta_E \\
 \ontop{\eqref{naturality of epsilon}}{=} {} & \bi{kqP(e)}{sf\varepsilon_XP(e)}t_{P(E),P(E)}P(\langle 1_{P(E)}, 1_{P(E)}\rangle)\delta_E \\
 = {} & \bi{k q}{sf\varepsilon_X}(P(e)+P(e))t_{P(E),P(E)}P(\langle 1_{P(E)}, 1_{P(E)}\rangle)\delta_E \\
 \ontop{\eqref{naturality of t}}{=} {} & \bi{k q}{sf\varepsilon_X}t_{P(X),P(X)} P(P(e)\times P(e))P(\langle 1_{P(E)}, 1_{P(E)}\rangle)\delta_E \\
 = {} & \bi{k q}{sf\varepsilon_X}t_{P(X),P(X)} P(\langle P(e), P(e)\rangle)\delta_E \\
 = {} & \bi{k q}{sf\varepsilon_X}t_{P(X),P(X)} P(\langle 1_{P(X)}, 1_{P(X)}\rangle) P^2(e)\delta_E \\
 \ontop{\eqref{naturality of delta}}{=} {} & \bi{k q}{sf\varepsilon_X}t_{P(X),P(X)} P(\langle 1_{P(X)}, 1_{P(X)}\rangle)\delta_X P(e)\\
 \ontop{\iS1}{=} {} & \varepsilon_X P(e)
 \ontop{\eqref{naturality of epsilon}}{=} e \varepsilon_E.	
\end{align*}

To prove \iS2 we use that $\widetilde{e}$ is a monomorphism
\begin{align*}
\widetilde{e}\gamma P(\bi{l}{\sigma})P(t_{L,W})\delta_{L\times W} & = q P(e)P(\bi{l}{\sigma})P(t_{L,W})\delta_{L\times W} \\
 & = q P(\bi{el}{e\sigma}t_{L,W})\delta_{L\times W} \\
 & = q P(\bi{k\widetilde{e}}{sw}t_{L,W})\delta_{L\times W} \\
 & = q P(\bi{k}{s}(\widetilde{e}+w)t_{L,W})\delta_{L\times W} \\
 & \ontop{\eqref{naturality of t}}{=} q P( \bi{k}{s} t_{K,Y} P(\widetilde{e}\times w))\delta_{L\times W} \\
 & = q P( \bi{k}{s}) P(t_{K,Y})P^2(\widetilde{e}\times w)\delta_{L\times W} \\
 & \ontop{\eqref{naturality of delta}}{=} q P( \bi{k}{s}) P(t_{K,Y})\delta_{K\times Y}P(\widetilde{e}\times w) \\
 & \ontop{\iS2}{=} \pi_K\varepsilon_{K\times Y}P(\widetilde{e}\times w) \\
 & \ontop{\eqref{naturality of epsilon}}{=} \pi_K (\widetilde{e}\times w) \varepsilon_{L\times W} \\
 & = \widetilde{e} \pi_L \varepsilon_{L\times W}.\qedhere
\end{align*}
\end{proof}

\begin{corollary}\label{closed lims in Pt_CC}
Intrinsic Schreier split extensions in $\C$ are closed under finite limits in $\Pt(\C)$.
\end{corollary}
\begin{proof}
By Corollary~\ref{exs of Schreier}, the terminal object of $\Pt(\C)$, the point $0\leftrightarrows 0$, is an intrinsic Schreier point. Then the result follows by Propositions~\ref{stability products} and~\ref{stability equalisers in Pt_CC}.
\end{proof}

\begin{corollary}\label{unital cat S-protomodular}
Any regular unital category with binary coproducts, functorial projective covers and a natural imaginary splitting is an $\s$-protomodular category with respect to the class $\s$ of intrinsic Schreier split epimorphisms.
\end{corollary}

It is easy to see that the closedness of $\SPt(\C)$ in $\Pt(\C)$ under finite limits implies that, for every object $Y$, the fibre $\SPt_Y(\C) $ is closed in $\Pt_Y(\C) $ under finite limits.

\section{(Intrinsic) Schreier special objects}\label{special}
In this section we investigate $\s$-special objects and the link between them and protomodular objects.

\begin{definition} \label{def special object}
An object $X$ of an $\s$-protomodular category is said to be \defn{$\s$-special} when the point
\begin{equation} \label{point S-special object}
\xymatrix@!0@=6em{ X \ar@{ |>->}[r]_-{\langle 1_X,0 \rangle} & X\times X \ar@<-.5ex>@{>>}[r]_-{\pi_2} & X \ar@<-.5ex>[l]_-{\langle 1_X,1_X \rangle}}
\end{equation}
(or, equivalently, the point $(\pi_1,\langle 1_X,1_X\rangle)$) belongs to the class $\s$.
\end{definition}

\begin{proposition}\cite[Proposition~6.2]{S-proto} \label{special objects protomodular core}
Given an $\s$-protomodular category~$\C$, the full subcategory of $\s$-special objects is a protomodular category, called the \defn{protomodular core} of $\C$ relative to the class $\s$.
\end{proposition}

If $\C$ is the category $\Mon$ of monoids, and $\s$ is the class of Schreier split epimorphisms, then the protomodular core is the category $\Gp$ of groups (\cite{S-proto}, Proposition $6.4$). On the other hand, in~\cite{2Chs} we showed that a monoid is a group if and only if it is a protomodular object. Thus one could be led to think that the notions of $\s$-special object and protomodular object are always equivalent, like in the case of monoids. This is false: actually, neither condition is implied by the other.

\begin{theorem}\label{Independence}
In J\'onsson--Tarski varieties, the concepts of a protomodular object and of an $\s$-special object (for $\s$ the class of Schreier split epimorphisms) are independent.
\end{theorem}

We start with a counterexample: a variety where not all $\s$-special objects are protomodular.

\begin{example}\label{Example Mag}
Let $\Mag$ be the J\'onsson--Tarski variety of unitary magmas (whose theory contains a unique constant $0$ and a binary operation $+$ such that $0+x=x=x+0$). We show that the unitary magma $Y=C_2=(\{0,1\},+)$, the cyclic group of order $2$, is an $\s$-special object which is not protomodular in $\Mag$.

It is easy to see that $Y$ is a special Schreier object: the map $q$ in
$$
\xymatrix@!0@=6em{Y \ar@<-.5ex>@{ |>->}[r]_-{\langle 1_Y,0 \rangle} & Y\times Y \ar@<-.5ex>@{>>}[r]_-{\pi_2} \ar@{-->}@<-.5ex>[l]_-{q} & Y \ar@<-.5ex>[l]_-{\langle 1_Y,1_Y \rangle}}
$$
which sends $(0,0)$ and $(1,1)$ to $0$, and $(1,0)$ and $(0,1)$ to $1$, is a Schreier retraction.

Next we show that $Y$ is not a protomodular object, by giving an example of a split epimorphism over $Y$ which is not a strong point. In $\Mag$ we consider the point $(f,s)$ and its kernel
\[
\xymatrix@!0@=4em{ \{0\} \ar@{^{(}->}[r]_-k & X \ar@<-.5ex>@{>>}[r]_-{f} & Y, \ar@<-.5ex>[l]_-{s}}
\]
where $X=(\{0,a,b\},+)$ is defined by
\[
\begin{array}{l|lll}
 + & 0 & a & b \\
 \hline
 0 & 0 & a & b \\
 a & a & 0 & 0 \\
 b & b & 0 & 0
\end{array}
\]
and
\[
\begin{cases}
f(0)=0\\
f(a)=1=f(b),	
\end{cases}
\qquad\qquad
\begin{cases}
s(0)=0\\
s(1)=a.	
\end{cases}
\]
Note that $X$ is not associative since, for instance, $(a+a)+b=0+b=b$ and $a+(a+b)=a+0=a$.

The point $(f,s)$ is not strong, because there exists a monomorphism $s$ which is not an isomorphism, yet makes the diagram
$$
\xymatrix{ & Y \ar@{ >->}[d]^-s \ar@{=}[dr] \\
 \{0\} \ar@{^{(}->}[ur] \ar@{^{(}->}[r]_-k & X & Y \ar[l]^-s}
$$
commute.
\end{example}

\begin{proposition} \label{right loop}
Let $\V$ be a J\'onsson--Tarski variety. An algebra in $\V$ is special with respect to the class of Schreier split epimorphisms if and only if it has a right loop structure.
\end{proposition}
\begin{proof}
We recall that a \defn{right loop} $(X, +, -, 0)$ is a set $X$ with two binary operations $+$ and $-$ and a unique constant $0$ such that the following axioms are satisfied:
\begin{equation*}\label{axioms right loop}
 x+0 = 0+x = x, \qquad (x-y)+y=x, \qquad (x + y) - y = x.
\end{equation*}
Now, given an object $X$ in $\V$, suppose that the split epimorphism~\eqref{point S-special object} is a Schreier split epimorphism. Then there exists a map $q \colon X \times X \dashrightarrow X$ such that, for all $x$, $y \in X$, we have
\[ (x, y) = (q(x,y), 0) + (y, y), \]
from which we deduce $x = q(x,y) +y$. Let us then define
\[ x - y = q(x, y). \]
Clearly we have $(x-y)+y = q(x,y) + y = x$. Moreover, \S2 tell us that
\[ q(x+y, y)= x, \]
or, in other words,
\[ (x+y)-y = x. \]
Conversely, given a right loop $(X, +, -, 0)$, we can define $q \colon X \times X \dashrightarrow X$ by putting $q(x,y) = x-y$. It is then immediate to check that \S1 and \S2 are satisfied.
\end{proof}

\begin{remark}\label{left loop}
In any J\'onsson--Tarski variety, a similar proof shows that an object~$X$ is special with respect to the class of left homogeneous split epimorphisms if and only if it has a \defn{left loop} structure:
\begin{equation*}\label{axioms left loop}
 x+0 = 0+x = x, \qquad x+(-x+y)=y, \qquad -x+(x+y) = y.
\end{equation*}
\end{remark}

As a consequence of these observations, we are able to complete the proof of Theorem~\ref{Independence} with Example~\ref{Heyting}.

We recall from~\cite{Bourn-Janelidze} that a pointed variety is protomodular if and only if its theory contains only one constant 0, and there exists a positive integer $n$, as well as $n$ binary operations $\alpha_1$, \dots, $\alpha_n$ and an $(n+1)$-ary operation $\beta$ such that
\[
\begin{cases}
\alpha_i(x,x)=0 &\forall i\in \{1,\dots, n\}, \\
 \beta(\alpha_1(x,y), \dots, \alpha_n(x,y), y)=x.
\end{cases}
\]
In many of the known examples of pointed protomodular varieties, such as $\Gp$ (or any variety which contains the right loop operations), the characterisation above works for $n=1$.

\begin{example}\label{Heyting}
Let $\HSLat$ be the variety of \defn{Heyting semilattices}, namely the variety defined by a unique constant $\top$ and two binary operations $\wedge$, $\rightarrow$ such that $(\wedge, \top)$ gives the structure of a semilattice with a top element, and the following axioms are satisfied:
\begin{align*}
(x\rightarrow x) &= \top \\
x\wedge (x\rightarrow y)&=x\wedge y \\
y\wedge (x\rightarrow y)&=y \\
x\rightarrow (y\wedge z)&=(x\rightarrow y)\wedge (x\rightarrow z).	
\end{align*}
In~\cite{Johnstone:Heyting} it was shown that $\HSLat$ is a protomodular variety. Accordingly, all objects in $\HSLat$ are protomodular. It was also shown that $\HSLat$ is protomodular for $n=2$, but not for $n=1$. This allows us to conclude that not all objects are $\s$-special. Indeed, $\HSLat$ is a J\'{o}nsson--Tarski variety, so by Proposition~\ref{right loop} all $\s$-special objects are right loops. If all objects were $\s$-special, then $\HSLat$ would be a protomodular variety with $n=1$, which is false.	
\end{example}

This argument can be used on any pointed protomodular variety for which the smallest number $n$ that makes the characterisation recalled above work is strictly larger than $1$. Hence our notion of intrinsic Schreier split epimorphism gives a categorical method for distinguishing general pointed protomodular varieties from varieties of right $\Omega$-loops. Indeed, in every protomodular variety $\V$ all objects are protomodular, while every object in $\V$ is $\s$-special with respect to the class $\s$ of (intrinsic) Schreier split epimorphisms if and only if $\V$ is a variety of right $\Omega$-loops.

\begin{proof}[Proof of Theorem~\ref{Independence}]
Combine Example~\ref{Example Mag} with Example~\ref{Heyting}.
\end{proof}


\providecommand{\noopsort}[1]{}
\providecommand{\bysame}{\leavevmode\hbox to3em{\hrulefill}\thinspace}
\providecommand{\MR}{\relax\ifhmode\unskip\space\fi MR }
\providecommand{\MRhref}[2]{%
  \href{http://www.ams.org/mathscinet-getitem?mr=#1}{#2}
}
\providecommand{\href}[2]{#2}

\end{document}